\documentclass[12pt]{amsart}
\usepackage{amscd,latexsym,amsthm,amsfonts,amssymb,amsmath,amsxtra,shapepar}
\usepackage[mathscr]{eucal}
\newcommand{\Irr}{\operatorname{Irr}}

\newcommand{\GL}{{\mathrm{GL}}}

\newcommand{\Hom}{{\mathrm{Hom}}}

\newcommand{\rk}{{\mathrm{k}}}

\newcommand{\Sp}{{\mathrm{Sp}}}

\newcommand{\vsp}{{\vspace{0.2in}}}

\newcommand{\con}{\textit{C}}

\newcommand{\Diff}{\operatorname{DO}}

\newcommand{\oM}{\operatorname{M}}
\newcommand{\oO}{\operatorname{O}}

\newcommand{\oSO}{\operatorname{SO}}
\newcommand{\oS}{\operatorname{S}}

\newcommand{\oT}{\operatorname{T}}
\newcommand{\oU}{\operatorname{U}}

\newcommand{\oD}{\textit{D}}

\newcommand{\g}{\mathfrak g}

\renewcommand{\l}{\mathfrak l}

\newcommand{\s}{\mathfrak s}

\renewcommand{\sl}{\mathfrak s \mathfrak l}
\newcommand{\gl}{\mathfrak g \mathfrak l}
\newcommand{\fM}{\mathfrak M}
\newcommand{\fZ}{\mathfrak Z}

\newcommand{\Z}{\mathbb{Z}}
\newcommand{\C}{\mathbb{C}}
\newcommand{\R}{\mathbb R}

\newcommand{\la}{\langle}
\newcommand{\ra}{\rangle}

\newcommand{\be}{\begin {equation}}
\newcommand{\ee}{\end {equation}}
\newcommand{\bee}{\begin {equation*}}
\newcommand{\eee}{\end {equation*}}

\theoremstyle{Theorem}

\theoremstyle{plain}
\newtheorem{thm}{Theorem}[section]
\newtheorem{cor}[thm]{Corollary}
\newtheorem{lem}[thm]{Lemma}
\newtheorem{prop}[thm]{Proposition}

\newtheorem{defn}[thm]{Definition}
\newtheorem{rmk}[thm]{Remark}

\title[Vanishing of quasi-invariant distributions]{Vanishing of quasi-invariant generalized functions}

\author [D. Jiang] {Dihua Jiang}
\address{School of Mathematics\\
University of Minnesota\\
206 Church St. S.E., Minneapolis\\
MN 55455, USA} \email{dhjiang@math.umn.edu}

\author [B. Sun] {Binyong Sun}
\address{Academy of Mathematics and Systems Science\\
Chinese Academy of Sciences\\
Beijing, 100190,  P.R. China} \email{sun@math.ac.cn}

\author [C.-B. Zhu] {Chen-Bo Zhu*}
\address{Department of Mathematics\\
National University of Singapore\\
Block S17, 10 Lower Kent Ridge Road\\
Singapore 119076} \email{matzhucb@nus.edu.sg}

\begin{document}


\thanks{*C.-B. Zhu thanks Professor Jae-Hyun Yang for the invitation to speak in International Conference on Geometry, Number Theory, and Representation Theory, October 10-12, 2012,
Inha University, Korea, as well as his hospitality during the conference}


\begin{abstract}
 Determination of quasi-invariant generalized functions is important for a variety of problems in representation theory, notably character theory and restriction problems.
 In this note, we review some new and easy-to-use techniques to show vanishing of quasi-invariant generalized functions, developed in the recent work of the authors (Uniqueness of Ginzburg-Rallis models: the Archimedean case, Trans. Amer. Math. Soc. 363, (2011), 2763-2802).  The first two techniques involve geometric notions attached to submanifolds, which we call metrical properness and unipotent $\chi$-incompatibility. The third one is analytic in nature, and it arises from the first occurrence phenomenon in Howe correspondence. We also highlight how these techniques quickly lead to two well-known uniqueness results,  on trilinear forms and Whittaker models.
\end{abstract}

\maketitle

\section{Quasi-invariant generalized functions}
\label{sec:inv-gen-func}

\subsection{Generalized functions}

Let $M$ be a smooth manifold.
Denote by $\con_0^\infty(M)$ the space of compactly supported
smooth functions on $M$, which is a complete
locally convex topological vector space under the usual inductive
smooth topology. Denote by $\oD^{-\infty}(M)$ the strong dual of
$\con_0^\infty(M)$, whose members are called distributions on $M$. A
distribution on $M$ is called a smooth density if under local
coordinate, it is a multiple of a smooth function with the
Lebesgue measure. Under the inductive smooth topology, the space
$\oD_0^\infty(M)$ of compactly supported smooth densities is again a
complete locally convex topological vector space, which is
(non-canonically) isomorphic to $\con_0^\infty(M)$. Denote by
$\con^{-\infty}(M)$ the strong dual of $\oD_0^\infty(M)$, whose
members are called generalized functions on $M$. By means of the integration pairing between functions and
densities, every smooth function (and more generally every locally integrable function) can be regarded as a generalized function.
The space $\con^{\infty}(M)$ of smooth functions in thus canonically and
continuously embedded in $\con^{-\infty}(M)$, and it has a dense image.

If $\phi:M\rightarrow M'$ is a smooth map of smooth manifolds,
then the pushing forward sends compactly supported distributions
on $M$ to compactly supported distributions on $M'$. If
furthermore $\phi$ is a submersion, then the pushing forward
induces a continuous linear map
\[
  \phi_*: \oD_0^\infty(M)\rightarrow \oD_0^\infty(M').
\]
We define the pulling back
\begin{equation}
\label{pullbackmap}
  \phi^*: \con^{-\infty}(M')\rightarrow \con^{-\infty}(M)
\end{equation}
as the transpose of $\phi_*$, which extends the usual pulling back
of smooth functions. The map $\phi^*$ is injective if $\phi$ is a
surjective submersion.

\begin{rmk} Pulling back is not canonically defined
for distributions. For this reason, we work with generalized
functions instead of distributions. Informally speaking,
``generalized functions transform like functions".
\end{rmk}

\subsection{Differential operators and transversality}

For $k\in \Z$, denote by $\Diff(M)_k$ the Fr\'{e}chet space of
differential operators on $M$ of order at most $k$, which by
convention is $0$ if $k<0$. It is well-known that every differential
operator $D:\con^{\infty}(M)\rightarrow \con^{\infty}(M)$ may be
continuously extended to $D:\con^{-\infty}(M)\rightarrow
\con^{-\infty}(M)$.

We have the principal symbol map
\[
   \sigma_k: \Diff(M)_k\rightarrow \Gamma^{\infty}(M,
   \oS^k(\oT(M)\otimes_\R  \C)),
\]
where $\oT(M)$ is the real tangent bundle of $M$, $\oS^k$ stands for
the $k$-th symmetric power, and $\Gamma^{\infty}$ stands for smooth
sections. The continuous linear map $\sigma_k$ is specified by the
following rule:
\[\begin{aligned}
\sigma_k(X_1 X_2\cdots X_k)(x)&=X_1(x) X_2(x)\cdots X_k(x), \quad \text{and}\\
&\sigma_k|_{\Diff(M)_{k-1}}=0,
\end{aligned}\]
for all $x\in M$ and all (smooth real) vector fields $X_1,X_2,
\cdots , X_k$ on $M$.

Let $Z$ be a (locally closed) submanifold of $M$. Write
\[
  \operatorname{N}_Z(M)=\oT(M)|_Z/\oT(Z)
\]
for the normal bundle of $Z$ in $M$. Denote by
\begin{equation}
\label{symbol}
   \sigma_{k,Z}: \Diff(M)_k\rightarrow \Gamma^{\infty}(Z,
   \oS^k(\operatorname{N}_Z(M)\otimes_\R \C))
\end{equation}
the map formed by composing $\sigma_k$ with the restriction map to
$Z$, and followed by the quotient map
\[
   \Gamma^{\infty}(Z, \oS^k(\oT(M)|_Z\otimes_\R  \C))\rightarrow \Gamma^{\infty}(Z,
   \oS^k(\operatorname{N}_Z(M)\otimes_\R \C)).
\]

\begin{defn}
\label{dtt}
\begin{itemize}
\item[(a)] {\rm A vector field $X$ on $M$ is said to be {\em tangential} to
$Z$ if $X(z)$ is in the tangent space $\oT_z(Z)$ for all $z\in Z$,
and {\em transversal} to $Z$ if $X(z)\notin \oT_z(Z)$ for all $z\in Z$;
more generally \item[(b)] a differential operator $D$ is said to
be {\em tangential} to $Z$ if for every point $z\in Z$ there is an open
neighborhood $U_z$ in $M$ such that $D|_{U_z}$ is a finite sum of
differential operators of the form $\varphi X_1 X_2\cdots X_r$,
where $\varphi$ is a smooth function on $U_z$, $r\geq 0$, and
$X_1,X_2,\cdots, X_r$ are vector fields on $U_z$ which are
tangential to $U_z\cap Z$. For $D\in \Diff(M)_k$, it is said to be
{\em transversal} to $Z$ if $\sigma_{k,Z}(D)$ does not vanish at any
point of $Z$.}
\end{itemize}
\end{defn}

We introduce some notations. For a locally closed subset $Z$ of
$M$, denote
\begin{equation}
  \con^{-\infty}(M;Z)=\{f\in\con^{-\infty}(U)|\ \text{supp} (f) \subseteq Z\},
\end{equation}
where $U$ is any open subset of $M$ containing $Z$ as a closed
subset. This definition is independent of $U$. For any
differential operator $D$ on $M$, denote
\begin{equation}
   \con^{-\infty}(M;Z; D)=\{f\in \con^{-\infty}(M;Z)|\ Df=0\}.
\end{equation}

The following proposition is due to Shalika \cite{Sh}. It asserts non-existence of
certain generalized functions with support in $Z$.

\begin{prop}\label{tvanishing2}
Let $D_1$ be a differential operator on $M$ of order $k>0$,
which is transversal to a submanifold $Z$ of $M$. Let $D_2$ be a
differential operator on $M$ which is tangential to $Z$. Then
\[
   \con^{-\infty}(M;Z;D_1+D_2)=0.
\]
\end{prop}

\begin{rmk} Shalika uses the transversality technique to show that certain generalized functions cannot be supported in
a Bruhat cell of strictly smaller dimension. See \cite[Proposition 2.10]{Sh}. This is a key step in his proof of uniqueness of Whittaker models
in the Archimedean case.
\end{rmk}

\subsection{Invariant generalized functions and restriction to a slice}
\label{inv}
Let $H$ be a Lie group, acting smoothly on a manifold $M$. Fix a
character $\chi$ on $H$. Denote by
\begin{equation}
  \con^{-\infty}_\chi(M)= \{f\in \con^{-\infty}(M)|\ f(hx)=\chi(h)f(x), \ \textrm{for }h\in
  H\}
\end{equation}
the space of $\chi $-equivariant generalized functions.

\vsp Let $\fM$ be a submanifold of $M$ and denote
\[
  \rho_{\fM}: H\times \fM\rightarrow M
\]
the action map.

\begin{defn}
\label{dslice} \begin{itemize} \item[(a)] {\rm We say that $\fM$ is a
{\em local $H$ slice} of $M$ if $\rho_{\fM}$ is a submersion, and an {\em $H$
slice} of $M$ if $\rho_{\fM}$ is a surjective submersion.
\item[(b)]
Given two submanifolds $\fZ\subset \fM$ of $M$, we say that $\fZ$ is
{\em relatively $H$ stable} in $\fM$ if
\[
  \fM\cap H \fZ=\fZ.
\]}
\end{itemize}
\end{defn}

Note that the relative stable condition amounts to saying that
$H\times \fZ$ is a union of fibres of the action map $\rho_{\fM}$.
The following lemma is elementary.

\begin{lem}\label{localization1}
Let $\fM$ be an $H$ slice of $M$, and let $\fZ$ be a relatively $H$
stable submanifold of $\fM$. Then $Z=H \fZ$ is a submanifold of $M$,
and $\fZ$ is an $H$ slice of $Z$. Furthermore if $\fZ$ is closed in
$\fM$, then $Z$ is closed in $M$.
\end{lem}

Now assume that $\fM$ is a local $H$ slice of $M$, and $H_{\fM}$
is a closed subgroup of $H$ which leaves $\fM$ stable. Let $H$ act
on $H\times \fM$ by left multiplication on the first factor, and
let $H_{\fM}$ act on $H\times \fM$ by
\[
   g(h,x)=(ghg^{-1}, gx), \quad g\in H_{\fM}, \, h\in H, x\in
   \fM.
\]
Then the submersion $\rho_{\fM}$ is $H$ intertwining as well as
$H_{\fM}$ intertwining. Therefore the pulling back yields a linear
map
\[
  \rho_{\fM}^*: \con^{-\infty}_\chi(M)\rightarrow \con^{-\infty}_\chi(H\times
 \fM)\cap \con^{-\infty}_{\chi_{\fM}}(H\times\fM),
\]
where $\chi_\fM=\chi|_{H_\fM}$. By the Schwartz Kernel Theorem \cite{Sc} and
the fact that every invariant distribution on a Lie group is a
scalar multiple of the Haar measure (\cite[8.A]{W1}), we have
\[
  \con^{-\infty}_\chi(H\times
 \fM)=\chi\otimes \con^{-\infty}(\fM).
\]
Consequently,
\[
  \con^{-\infty}_\chi(H\times
 \fM)\cap \con^{-\infty}_{\chi_{\fM}}(H\times\fM)=\chi\otimes
 \con^{-\infty}_{\chi_{\fM}}(\fM).
\]

We thus have the following

\begin{prop}\label{restriction}
There is a well-defined map which is called the restriction to
$\fM$:
\[
  \con^{-\infty}_\chi(M)\rightarrow
  \con^{-\infty}_{\chi_{\fM}}(\fM),\quad f\mapsto f|_\fM
\]
by requiring that
\[
  \rho_{\fM}^*(f)=\chi\otimes f|_\fM.
\]
The map is injective when $\fM$ is an $H$ slice.
\end{prop}

\begin{rmk} Aizenbud and Gourevitch \cite{AG} have developed sophisticated techniques (called generalized Harish-Chandra descent) for working with $G$-invariant generalized functions on a smooth affine $G$-variety, based on Luna's slice theorem \cite{Lu}.
\end{rmk}

\section{Metrical properness and unipotent $\chi$-incompatibility}
\label{sec:MPUI}

\subsection{Metrical properness}

This notion requires that the manifold $M$ is pseudo Riemannian,
i.e., the tangent spaces are equipped with a smoothly varying
family $\{\la\,,\,\ra_x:x\in M\}$ of nondegenerate symmetric
bilinear forms.

\begin{defn}
\label{dmp} \begin{itemize} \item[(a)] {\rm A submanifold $Z$ of a
pseudo Riemannian manifold $M$ is said to be {\em metrically proper} if
for all $z\in Z$, the tangent space $\oT_z(Z)$ is contained in a
proper nondegenerate subspace of $\oT_z(M)$. \item[(b)] A
differential operator $D\in \Diff(M)_2$ is said to be of {\em Laplacian
type} if for all $x\in M$, the principal symbol
\[
  \sigma_2(D)(x)=u_1 v_1+u_2 v_2+\cdots+u_m v_m,
\]
where $u_1,u_2,\cdots,u_m$ is a basis of the tangent space
$\oT_x(M)$, and $v_1,v_2,\cdots,v_m$ is the dual basis in $\oT_x(M)$
with respect to $\la\,,\,\ra_x$.}
\end{itemize}
\end{defn}

Note that a Laplacian type differential operator is transversal to
any metrically proper submanifold, from its very definition.
Therefore the following is a special case of Proposition
\ref{tvanishing2}.

\begin{prop}\label{tvanishing3}
Let $Z$ be a metrically proper submanifold of $M$, and let $D$ be a
Laplacian type differential operator on $M$. Then
\[
   \con^{-\infty}(M;Z;D)=0.
\]
\end{prop}

\begin{rmk} The above proposition may be viewed as a form of uncertainty principle.  The second and third named authors exploited metrical properness in their proof of Archimedean multiplicity one theorems \cite[Section 5]{SZ2}.
\end{rmk}

\subsection{Unipotent $\chi$-incompatibility}

As in Section \ref{inv}, let $H$ be a Lie group with a character
$\chi$ on it, acting smoothly on a manifold $M$. If a locally
closed subset $Z$ of $M$ is $H$-stable, denote by
$\con^{-\infty}_\chi(M;Z)$ the space of all $f$ in
$\con^{-\infty}(M;Z)$ which are $\chi $-equivariant. We shall use
similar notations (such as $\con^{-\infty}_\chi(M;D)$ and
$\con^{-\infty}_\chi(M;Z;D)$) without further explanation.

\begin{defn}
\label{duc} {\rm An $H$-stable submanifold $Z$ of $M$ is said to be
{\em unipotently $\chi$-incompatible} if for every $z_0\in Z$, there is a
local $H$ slice $\fZ$ of $Z$, containing $z_0$, and a smooth map
$\phi: \fZ\rightarrow H$ such that the followings hold for all $z\in
\fZ$:
\begin{itemize}
 \item[(a)]
   $\phi(z)z=z$, and
 \item[(b)]
   the linear map
   \[
     \oT_z(M)/\oT_z(Z)\rightarrow \oT_z(M)/\oT_z(Z)
   \]
      induced by the action of $\phi(z)$ on $M$ is unipotent;
      \item[(c)]
$\chi(\phi(z))\neq 1$.
\end{itemize}}
\end{defn}

The notion of unipotent $\chi$-incompatibility is of importance due to the following

\begin{prop}\label{localization3}
Let $Z$ be an $H$-stable submanifold of $M$ which is unipotently
$\chi$-incompatible. Then $\con_\chi^{-\infty}(M;Z)=0$.
\end{prop}

\begin{rmk} Aizenbud, Offen and Sayag have proved an analog of the above proposition in an algebraic setting.
See \cite[Proposition 2.3]{AOS}.
\end{rmk}

\subsection{A synthesis: $\oU_{\chi}\oM$ property}
As before, let $H$ be a Lie group acting smoothly on a manifold
$M$, and let $\chi$ be a character on $H$. We further assume that
$M$ is a pseudo Riemannian manifold.

\begin{defn}
\label{dum} We say that an $H$-stable locally closed subset $Z$ of
$M$ has $\oU_{\chi}\oM$ property if there is a finite filtration
\[
  Z=Z_0\supset Z_1\supset\cdots\supset Z_k\supset Z_{k+1}=\emptyset
\]
of $Z$ by $H$-stable closed subsets of $Z$ such that each
$Z_i\setminus Z_{i+1}$ is a submanifold of $M$ which is either
unipotently $\chi$-incompatible or metrically proper in $M$.
\end{defn}

As a combination of Propositions \ref{tvanishing3} and \ref{localization3}, we have
\begin{prop}
\label{UchiM} Let $D$ be a  differential operator on $M$ of
Laplacian type. Let $Z$ be an $H$-stable closed subset of $M$
having $\oU_{\chi}\oM$ property. Then
\[
  \con^{-\infty}_\chi(M;Z;D)=0.
\]
\end{prop}

\begin{rmk} Subsets satisfying $\oU_{\chi}\oM$ property indeed arise in applications. See \cite[Sections 5 and 6]{JSZ}, as an example.
\end{rmk}

\section{First occurrence in Howe correspondence}
\label{sec:FirstOccur}

First occurrence phenomenon in (local and global) theta correspondence was discovered by S.S. Kudla \cite{Ku1} and S. Rallis \cite{Ra}. In the mid 1990's, Kudla and Rallis put forward their conservation conjectures on first occurrences in local theta (or Howe) correspondence, and their pioneering work have profound implications. We refer the reader to \cite{KR}. The full conjectures are now proved, in a 2012 preprint by the second and third named authors \cite{SZ3}.

We shall only be concerned with the orthogonal-symplectic dual pair and the first occurrences of two very special characters. Note that in the dual pair setting, non-occurrence of characters amounts to vanishing of quasi-invariant generalized functions. We shall focus on the Archimedean case.

Let $\rk =\R$ or $\C$. Let $V$ be a (non-degenerate) quadratic space and $W$ be a symplectic space, over $\rk$, and consider the reductive dual pair \cite{Ho79}:
\[(\oO(V), \Sp(W))\subset \Sp(V\otimes W).\]
Fix a nontrivial unitary character $\psi $ of $\rk$, and consider the smooth oscillator representation $\omega _{V,W}$ associated to the dual pair $(\oO(V), \Sp(W))$ and to the character $\psi$.

We fix a parity $\epsilon\in \Z/2\Z$ of $\dim V$. Write
\begin{equation}\label{meta}
   1\rightarrow \{\pm 1\}\rightarrow
   {\Sp}_\epsilon(W)\rightarrow
   \Sp(W)\rightarrow 1
\end{equation}
for the unique topological central extension of the symplectic group
$\Sp(W)$ by $\{\pm 1\}$ such that it does not split if $\rk$ is isomorphic to $\R$, $\epsilon$ is odd, and $W$ is nonzero,
and it splits otherwise. It is well-known that the oscillator representation $\omega _{V,W}$ yields a representation of
$\oO(V)\times {\Sp}_\epsilon(W)$.

Denote by $\Irr(\oO(V))$ the isomorphism classes of
irreducible Casselman-Wallach representations of $\oO(V)$,
and $\Irr({\Sp}_\epsilon(W))$ the isomorphism
classes of irreducible genuine Casselman-Wallach representations of
${\Sp}_\epsilon(W)$. The reader may consult \cite[Chapter
11]{W2} for details about Casselman-Wallach representations.
Throughout this section, $\pi$ denotes a representation in $\Irr(\oO(V))$ and $\rho$
denotes a representation in $\Irr(\Sp_\epsilon(W))$. We are
interested in occurrences of $\pi$ and $\rho$ in the local theta
correspondence \cite{Ho89}.

\subsection{Orthogonal group}

We first consider the case of orthogonal groups. Thus we fix $V$. Recall Kudla's persistence
principle \cite{Ku2}: if $W_1,W_2$ are two symplectic spaces, and $\dim W_1\leq
\dim W_2$, then
\[
 \Hom_{\oO(V)}(\omega_{V,W_1},
\pi)\neq 0\quad\textrm{implies} \quad \Hom_{\oO(V)}(\omega_{V,W_2},
\pi)\neq 0.
\]

Define the first occurrence index
\begin{equation}
\label{Orth-index}
  \operatorname n(\pi):=\min\{\frac{1}{2}\dim W\mid \Hom_{\oO(V)}(\omega_{V,W},
\pi)\neq 0\}.
\end{equation}

The following result is due to Przebinda \cite[Theorem C.7]{Prz}.

\begin{prop}
\label{stableO} We have
\[\operatorname n(\det)=\dim V,\]
where $\det$ stands for the determinant character of $\oO(V)$.
\end{prop}

\begin{cor}
\label{first-occur}
Let $V$ be a finite dimensional non-degenerate
quadratic space over $\rk$, and let the orthogonal group $\oO(V)$ act
on $V^n$ diagonally, where $n$ is a positive integer. If $n<\dim V$,
and if a tempered generalized function $f$ on $V^n$ is
$\oSO(V)$-invariant, then $f$ is $\oO(V)$-invariant.
\end{cor}

\subsection{Symplectic group}

Now we consider the case of symplectic groups. The results to be discussed are of a slightly different nature than the rest of this note, because the ${\Sp}_\epsilon(W)$ representation is not induced from a geometric action. Nevertheless, we shall present the results parallel to the case of orthogonal groups.

We fix $W$. Kudla's persistence principle \cite{Ku2} says that if $V_1,V_2$ are two quadratic spaces belonging to the same Witt tower,
and $\dim V_1\leq \dim V_2$, then
\[
 \Hom_{\Sp_\epsilon(W)}(\omega_{V_1,W},
\rho)\neq 0\quad\textrm{implies} \quad
\Hom_{\Sp_\epsilon(W)}(\omega_{V_2,W}, \rho)\neq 0.
\]

We also fix a Witt tower $T$ of quadratic spaces. Define the first occurrence index
\begin{equation}
\label{Symp-index}
   \operatorname m_T(\rho):=\min\{\dim V\mid V\in T,\,\Hom_{\Sp_\epsilon(W)}(\omega_{V,W},
\rho)\neq 0\}.
\end{equation}

The following result is due to Loke \cite[Theorem 1.2.1]{LL}, and it amounts to the determination of
$m_T(C)$, where $\C$ stands for the unique one-dimensional genuine representation of $\Sp_\epsilon(W)$ (when $\dim V$ is even).
Recall that a quadratic space $V$ is called quasi-split if its split rank $\geq \frac{\dim V-2}{2}$. Write $\dim W =2n$.

\begin{prop}\label{stableSp}
Assume that $\epsilon$ is even.
\begin{itemize}
\item[(i)] If $V$ is not quasi-split, then
\[
 \Hom_{\Sp_\epsilon(W)}(\omega_{V,W},\C)\neq 0
\]
implies that $V$ has split rank $\geq 2n$, in particular $\dim V\geq
4n+4$.
\item[(ii)] If $\rk =\R$ and $V$ is quasi-split and nonsplit, then
\[
 \Hom_{\Sp_\epsilon(W)}(\omega_{V,W},\C)\neq 0
\]
implies that $V$ has split rank $\geq n$, in particular $\dim V\geq 2n+2$.
\end{itemize}
\end{prop}

\begin{rmk} The proof of the above proposition is through $K$-type computations. See \cite[Section 3]{LL}. It is worth mentioning that Propositions \ref{stableO} and \ref{stableSp} have their non-Archimedean analogs \cite{Ra,KR}, and a uniform proof of these analogs is given in \cite[Section 5]{SZ3}.
\end{rmk}

\section{Two quick applications}

As applications of the techniques discussed in this note, we present two quick and sweet consequences. See \cite[Section 11]{JSZ}.

\subsection{Uniqueness of trilinear forms}
\label{first-application}

Let $\rk =\R$ or $\C$. Let $H_{2}=\GL_2(\rk)$, and $\tilde{H}_{2}$ be the following extended group:
\[
  \tilde{H}_{2}=\{1,\tau\}\ltimes \GL_2(\rk),
\]
where the semidirect product is given by the action
\[
  \tau(g)=g^{-t}.
\]
Denote by $\tilde{\chi}_2$ the character of $\tilde{H}_2$ such
that
\[
  \tilde{\chi}_2|_{\GL_2(\rk)}=1\quad\textrm{and}\quad \tilde{\chi}_2(\tau)=-1.
\]

As in \cite{JSZ}, we use $\con^{-\xi}$ to denote (appropriate) space of tempered generalized functions.

\begin{lem}\label{m2}
Let $\tilde{H}_2$ act on $M_2=\GL_2(\rk)\times \GL_2(\rk)$ by
\[
  g(x,y)=(gxg^{-1},gyg^{-1}),\quad g\in \GL_2(\rk),
\]
and
\[
  \tau(x,y)=(x^t,y^t).
\]
Then
\[
  \con^{-\xi}_{\tilde{\chi}_2}(M_2)=0.
\]
\end{lem}
\begin{proof}
Using the same formula, we may extend the action of
$\tilde{H}_2$ on $\GL_2(\rk)\times \GL_2(\rk)$ to the larger space
$\gl_2(\rk)\times \gl_2(\rk)$. It suffices to
prove that
\[
   \con^{-\xi}_{\tilde{\chi}}(\gl_2(\rk)\times \gl_2(\rk))=0.
\]

Identify $\rk$ with the center of $\gl_2(\rk)$. We have
\[
   \gl_2(\rk)\times \gl_2(\rk)=(\s\l_2(\rk)\times \s\l_2(\rk))\oplus
   (\rk\times \rk)
\]
as a $\rk$ linear representation of $\tilde{H}_2$, where
$\tilde{H}_2$ acts on $\rk\times \rk$ trivially. Therefore it
suffices to prove that
\[
   \con^{-\xi}_{\tilde{\chi}_{2}}(\s\l_2(\rk)\times \s\l_2(\rk))=0.
\]
We view $\sl_2(\rk)$ as a three dimensional quadratic space under
the trace form. Under this identification, the action of
$\tilde{H}_{2}$ yields the diagonal action of $\oO(\sl_2(\rk))$ on
$\sl_2(\rk)\times \sl_2(\rk)$, with $\tilde{\chi}_{2}$ corresponding
to the determinant character. So the required vanishing result is
a special case of Corollary \ref{first-occur}.
\end{proof}

The following theorem is proved in \cite{Lo} (in an exhaustive
approach), and its p-adic analog was proved much earlier in
\cite[Theorem 1.1]{Pra}.

\begin{thm}\label{Whittaker2}
Let $V$ be an irreducible Casselman-Wallach representation of
$\GL_{2}(\rk)\times \GL_{2}(\rk)\times\GL_{2}(\rk)$. Then
\[
   \dim \Hom_{\GL_{2}(\rk)} (V,\C_{\chi_{2}})\leq 1.
\]
Here we view $\GL_{2}(\rk)$ as the diagonal subgroup of
$\GL_{2}(\rk)\times \GL_{2}(\rk)\times\GL_{2}(\rk)$, $\chi
_2=\chi_{\rk^\times}\circ \det$ is a character of $\GL_{2}(\rk)$, and
$\chi_{\rk^\times}$ is an arbitrary character of $\rk^\times$.
\end{thm}

\begin{proof} By the Gelfand-Kazhdan criterion \cite{GK} (a general form is in \cite[Theorem 2.3]{SZ1}),
one just needs to show
the following: let $\GL_2(\rk)\times \GL_2(\rk)$ act on
\begin{equation*} \label{dm2}
  G_{2,2,2}=\GL_2(\rk)\times \GL_2(\rk)\times \GL_2(\rk)
\end{equation*}
by
\[
  (g_1,g_2)(x,y,z)=(g_1 x g_2^t,g_1 y g_2^t, g_1 z g_2^t),\quad
  g_1,g_2\in \GL_2(\rk).
\]
Denote by $\chi_{2,2}$ the character of $\GL_2(\rk)\times \GL_2(\rk)$
given by
\[
  \chi_{2,2}(g_1,g_2)=\chi_{\rk^\times}(\det(g_1))\chi_{\rk^\times}(\det(g_2)),
  \quad g_1,g_2\in \GL_2(\rk).
  \]
Then for all $f\in \con^{-\xi}_{\chi_{2,2}}(G_{2,2,2})$, we have
\[f(x^{t},y^{t},z^{t})=f(x,y,z).
\]
To show the above, we observe that
$M_2\cong \GL_2(\rk)\times\GL_2(\rk)\times \{I_2\}$ is a $\GL_2(\rk)\times
\GL_2(\rk)$ slice of $G_{2,2,2}$, which is stable under $H_2\cong \{(x,
x^{-t})\mid x\in \GL_2(\rk)\}\subset \GL_2(\rk)\times \GL_2(\rk)$
and $\tau$. The result then follows from Lemma \ref{m2}, in view of Proposition \ref{restriction}.
\end{proof}

\subsection{Uniqueness of Whittaker models}
\label{second-application}

Let $\mathbf{G}$ be a quasisplit connected reductive algebraic group
defined over $\R$. Let $\mathbf{B}$ be a Borel subgroup of
$\mathbf{G}$, with unipotent radical $\mathbf{N}$. Let
\[
  \chi_\mathbf{N}:\mathbf{N}(\R)\rightarrow \C^\times
\]
be a generic unitary character. The meaning of ``generic" will be
explained later in the proof.

The following theorem is fundamental and well-known. For
$\mathbf{G}=\GL_n$, this is a celebrated result of Shalika
\cite{Sh}. A proof in general may be found in \cite[Theorem
9.2]{CHM}. We shall give a short proof based on the notion of
unipotent $\chi$-incompatibility.

\begin{thm}
Let $V$ be an irreducible Casselman-Wallach representation of
$\mathbf{G}(\R)$. Then
\[
   \dim \Hom_{\mathbf{N}(\R)} (V,\C_{\chi_{\mathbf{N}}})\leq 1.
\]
\end{thm}

\begin{proof}
Define a Casselman-Wallach distributional representation
to be the strong dual of a Casselman-Wallach representation. The current
theorem can then be reformulated as follows: the space
\[
 U^{\chi_\mathbf{N}^{-1}}=\{u\in U\mid gu=
 \chi_\mathbf{N}^{-1}(g)u\,\textrm{ for all } g\in \mathbf{N}(\R)\}
\]
is at most one dimensional for every irreducible Casselman-Wallach
distributional representation $U$ of $\mathbf{G}(\R)$.

Let $\bar{\mathbf{B}}$ be a Borel subgroup opposite to $\mathbf{B}$,
with unipotent radical $\bar{\mathbf{N}}$. Then
$\mathbf{T}=\mathbf{B}\cap \bar{\mathbf{B}}$ is a maximal torus. Let
\[
  \chi_\mathbf{T}:\mathbf{T}(\R)\rightarrow \C^\times
\]
be an arbitrary character. Then
\[
\begin{aligned}
  U(\chi_\mathbf{T})=&\{f\in\con^{-\infty}(\mathbf{G}(\R))\mid f(t\bar{n}x)
  =\chi_\mathbf{T}(t)f(x)\, \\
     & \textrm{ for all } t\in \mathbf{T}(\R), \bar{n}\in \bar{\mathbf{N}}(\R)\}
\end{aligned}
\]
is the distributional version of nonunitary principal series
representations. By Casselman's subrepresentation theorem (in the
category of Casselman-Wallach distributional representations), it
suffices to show that
\begin{equation}\label{dimu}
  \dim U(\chi_\mathbf{T})^{\chi_\mathbf{N}^{-1}}\leq 1, \quad
  \text{for any $\chi_\mathbf{T}$}.
\end{equation}

Let
\[
   H_\mathbf{G}=\bar{\mathbf{B}}(\R)\times \mathbf{N}(\R),
\]
which acts on $\mathbf{G}(\R)$ by
\[
   (\bar{b}, n)x=\bar{b}xn^{-1}.
\]
Write
\[
  \chi_\mathbf{G}(t\bar{n},n)=\chi_\mathbf{T}(t) \chi_\mathbf{N}(n),
\]
which defines a character of $H_\mathbf{G}$. Then (\ref{dimu}) is
equivalent to
\begin{equation}\label{dimcong}
   \dim \con^{-\infty}_{\chi_\mathbf{G}}(\mathbf{G}(\R))\leq 1.
\end{equation}

Let $W$ be the Weyl group of $\mathbf{G}(\R)$ with respect to
$\mathbf T$.
We have the Bruhat decomposition
\[
  \mathbf{G}(\R)=\bigsqcup_{w\in W} \mathbf{G}_w,\quad
  \textrm{with}\quad
  \mathbf{G}_w=\bar{\mathbf{B}}(\R)w \mathbf{N}(\R).
\]
From this we form a $H_\mathbf{G}$ stable filtration
\[
  \emptyset=\mathbf{G}^0\subset
  \mathbf{G}^1\subset
  \mathbf{G}^2\subset\cdots\subset \mathbf{G}^r=\mathbf{G}(\R)
\]
of $\mathbf{G}(\R)$ by open subsets, with
$\mathbf{G}^1=\bar{\mathbf{B}}(\R)\mathbf{N}(\R)$ and every
difference $\mathbf{G}^i\setminus \mathbf{G}^{i-1}$ a Bruhat cell
$\mathbf{G}_w$, for $i\geq 2$.

Clearly (\ref{dimcong}) is implied by the following two assertions:
\begin{equation}\label{dimcong2}
  \dim \con^{-\infty}_{\chi_\mathbf{G}}(\mathbf{G}^1)=1;
\end{equation}
and
\begin{equation}\label{dimcong3}
  \textrm{if $f\in
\con^{-\infty}_{\chi_\mathbf{G}}(\mathbf{G}^{i})$ vanishes on
$\mathbf{G}^{i-1}$, then $f=0$,}
\end{equation}
for $i\geq 2$. The equality (\ref{dimcong2}) is clear as
$\mathbf{G}^1=\bar{\mathbf{B}}(\R)\mathbf{N}(\R)$. For
(\ref{dimcong3}), we write
\[
  \mathbf{G}^i\setminus \mathbf{G}^{i-1}=\mathbf{G}_w, \quad
  \textrm{with $w$ a non-identity element of $W$}.
\]
The genericity means that $\chi_\mathbf{N}$ has nontrivial
restriction to $\mathbf{N}(\R)\cap w^{-1}(\bar{\mathbf{N}}(\R))w$.
Pick
\[
 n=w^{-1}\bar{n} w\in \mathbf{N}(\R)\cap
w^{-1}(\bar{\mathbf{N}}(\R))w
\]
so that $\chi_\mathbf{N}(n)\neq 1$. Then $(\bar{n},n)\in
H_\mathbf{G}$ satisfies
\[
  (\bar{n},n)w=w, \quad \textrm{and  }
  \chi_\mathbf{G}(\bar{n},n)=\chi_{\mathbf{N}}(n)\neq 1.
\]
Consequently, $\mathbf{G}_w$ is unipotently
$\chi_\mathbf{G}$-incompatible. Now (\ref{dimcong3}) follows from
Proposition \ref{localization3}.
\end{proof}

\vsp

\noindent {\bf Acknowledgements}: Dihua Jiang is supported in part by NSF (USA) grant
DMS--0653742 and by a Distinguished Visiting Professorship at the
Academy of Mathematics and System Sciences, the Chinese Academy of
Sciences. Binyong Sun is supported by NSFC grants 10931006 and 11222101.
Chen-Bo Zhu is supported by NUS-MOE grant MOE2010-T2-2-113.

\end{document}